\documentclass[12pt]{article}
\usepackage{amssymb,amsmath,amsthm}

\def\ab{{\bold a}}

\def\xb{{\bold x}}
\def\yb{{\bold y}}

\def\0b{{\bold 0}}

\def\ab{{\bold a}}

\def\xb{{\bold x}}
\def\yb{{\bold y}}

\def\0b{{\bold 0}}

\usepackage{bm}
\bmdefine{\Bzero}{0}
\bmdefine{\Bone}{1}

\def\Bone{{\bf 1}}

\def\a{{\bf a}}

\def\RR{{\mathbb R}}
\def\ZZ{{\mathbb Z}}

\def\QQ{{\mathbb Q}}

\newtheorem{Theorem}{Theorem}[section]

\newtheorem{Corollary}[Theorem]{Corollary}
\newtheorem{Proposition}[Theorem]{Proposition}
\newtheorem{Remark}[Theorem]{Remark}

\newtheorem{Example}[Theorem]{Example}

\newtheorem{Conjecture}[Theorem]{Conjecture}

\def\cpoly{{\rm Cut}^\square(G)}
\def\cpolytope{{\rm Cut}^\square}

\title{Normality of cut polytopes of graphs is a minor closed property}
\author{Hidefumi Ohsugi}
\date{}

\begin{document}

\maketitle

\begin{abstract}
Sturmfels--Sullivant conjectured that the cut polytope of a graph
is normal if and only if the graph has no $K_5$ minor.
In the present paper, it is proved that
the normality of cut polytopes of graphs is a minor closed property.
By using this result, we have large classes of normal cut polytopes. 
Moreover, it turns out that, in order to study the conjecture,
it is enough to consider 4-connected plane triangulations.
\end{abstract}

\section{Introduction}

Let $G$ be a graph on the vertices $[n]:= \{1,2,\ldots,n\}$
and edges $E$
without loops or multiple edges.
Let $S \subset [n]$.
Then the cut semimetric on $G$ induced by $S$ is the $0/1$ vector
$\delta_G(S)$ in $\RR^E$ defined by 
$$
\delta_G(S)_{ij} =
\left\{
\begin{array}{cc}
 1  & \mbox{ if } | S \cap \{i,j\}|=1\\
 0 &  \mbox{ otherwise}
\end{array}
\right.
$$
where $ij \in E$.
Let 
$A_G= \{\a_1,\ldots,\a_N\} = \{\delta_G (S) \ | \ S \subset [n]\} \subset \ZZ^E$
where $N=2^{n-1}$.
The {\em cut polytope} $\cpoly$ of $G$ is the convex hull of $A_G$.
Let
\begin{eqnarray*}
X_G &:=&
\left\{
\left(
\begin{array}{c}
\a_1\\
1
\end{array}
\right)
,\ldots,
\left(
\begin{array}{c}
\a_N\\
1
\end{array}
\right)
\right\} \subset \ZZ^{E+1}, \\
\ZZ (X_G) &:=& \left\{ \left. \sum_{i=1}^N z_i 
\left(
\begin{array}{c}
\a_i\\
1
\end{array}
\right)
 \ \right| \   z_i \in \ZZ \right\} \subset \ZZ^{E+1},\\
\QQ_+ (X_G) &:=& \left\{ \left. \sum_{i=1}^N q_i 
\left(
\begin{array}{c}
\a_i\\
1
\end{array}
\right)
 \ \right| \ 0 \leq q_i \in \QQ \right\} \subset \QQ^{E+1},\\
\ZZ_+ (X_G) &:=& \left\{ \left. \sum_{i=1}^N z_i 
\left(
\begin{array}{c}
\a_i\\
1
\end{array}
\right)
 \ \right| \ 0 \leq z_i \in \ZZ \right\} \subset \ZZ^{E+1}.
\end{eqnarray*}
Then $\ZZ_+ (X_G) \subset \ZZ (X_G) \cap \QQ_+ (X_G)$ holds in general.
The cut polytope $\cpoly$ is called {\em normal}
if we have $\ZZ_+ (X_G) = \ZZ (X_G) \cap \QQ_+ (X_G)$.

\subsection{A conjecture on normal cut polytopes}

Let
$K[{\bf t},s]=K[t_1,\ldots,t_E,s]$ be
the polynomial ring in $E+1$ variables over a field $K$
and let
$K[{\bf q}]=K[q_1,\ldots,q_N]$
the polynomial ring in $N (= 2^{n-1})$ variables over $K$.
For each nonnegative integer vector ${\bf \alpha} =(\alpha_1,\ldots,\alpha_E)
\in \ZZ^E$, we set
${\bf t}^\alpha = t_1^{\alpha_1} \cdots t_E^{\alpha_E} $.
Then the {\em toric cut ideal} $I_G$ of a graph $G$ is 
the kernel of homomorphism
$
\pi : K[{\bf q}] \longrightarrow  K[{\bf t},s]
$
defined by $\pi(q_i) = {\bf t}^{{\bf a}_i} s$.
Sturmfels--Sullivant \cite[Conjecture 3.7]{StSu} conjectured 
that $K[{\bf q}]/ I_G$ is normal if and only if $G$ has no $K_5$ minor.
Since it is known (e.g., \cite[Proposition 13.5]{Stu}) that 
$K[{\bf q}]/ I_G$ is normal if and only if 
$\ZZ_+ (X_G) = \ZZ (X_G) \cap \QQ_+ (X_G)$ holds,
their conjecture is formulated as follows:

\begin{Conjecture}
\label{concon}
{\em
The cut polytope $\cpoly$ is normal if and only if
$G$ has no $K_5$ minor.
}
\end{Conjecture}

If $\cpoly$ is normal
and $G'$ is obtained from $G$ by contracting an edge,
then $\cpolytope (G')$ is normal (\cite[Lemma 3.2 (2)]{StSu}).
Note that, if a graph $G$ has $K_m$ as a minor,
then that minor can be realized by a sequence of edge contraction only.
As stated in \cite{StSu}, the ``only if" part is true since $\cpolytope (K_5)$ is not normal.
On the other hand, the ``if" part is true for the following classes
of graphs:

\begin{itemize}
\item
graphs with $\leq 6$ vertices (by a direct computation \cite{StSu}
together with \cite[Theorem 1.2]{StSu})
\item
graphs having no induced cycle of length $\geq 5$
(by \cite[Theorem 3.2]{Sul})
\item
``ring graphs" (Note that ring graphs have no $K_4$ minor. See \cite{NaPe}).
\end{itemize}

\subsection{Hilbert bases of cut polytopes}

In order to avoid confusion,
we must introduce ``nonhomogeneous" version of this problem on cut polytopes.
The following sets are studied in, e.g., \cite{Fugo,Lau}:
\begin{eqnarray*}
\ZZ (A_G) &:=& \left\{ \left. \sum_{i=1}^N z_i 
\a_i
\ \right| \   z_i \in \ZZ \right\} \subset \ZZ^{E}\\
\QQ_+ (A_G) &:=& \left\{ \left. \sum_{i=1}^N q_i 
\a_i
\ \right| \ 0 \leq q_i \in \QQ \right\} \subset \QQ^{E}\\
\ZZ_+ (A_G) &:=& \left\{ \left. \sum_{i=1}^N z_i 
\a_i
\ \right| \ 0 \leq z_i \in \ZZ \right\} \subset \ZZ^{E}
\end{eqnarray*}
If $\ZZ_+ (A_G) = \ZZ (A_G) \cap \QQ_+ (A_G) $ holds,
then $A_G$ is called a {\em Hilbert basis}.
It is known that
$\ZZ_+ (A_G) = \ZZ (A_G) \cap \QQ_+ (A_G) $
holds if one of the following holds:
\begin{itemize}
\item
$G$ has no $K_5$ minor (\cite[Corollary 1.3]{Fugo});
\item
$G$ is $K_6 \setminus e$ or its subgraph (\cite[Theorem 1.1]{Lau}).
\end{itemize}
Moreover, $\ZZ_+ (A_G) \neq \ZZ (A_G) \cap \QQ_+ (A_G) $
holds if
\begin{itemize}
\item
$G$ has $K_6$ minor (\cite[Proposition 1.2]{Lau}).
\end{itemize}
On the other hand, it is known that the class of graphs $G$ satisfying
$\ZZ_+ (A_G) = \ZZ (A_G) \cap \QQ_+ (A_G) $
is closed under
\begin{itemize}
\item
contraction minors (\cite[Proposition 2.1]{Lau});
\item
clique sums (\cite[Proposition 2.7]{Lau});
\item
edge deletions satisfying some conditions (\cite[Proposition 2.3]{Lau}).
\end{itemize}

Hence it is natural to have the following conjecture.

\begin{Conjecture}
Let $G$ be a connected graph.
Then $
\ZZ_+ (A_G )
= 
\ZZ (A_G) \cap \QQ_+ (A_G) 
$
if and only if $G$ has no $K_6$ minor.
\end{Conjecture}

The relation between our problem and this problem is as follows:

\begin{Proposition}
\label{nonhomo}
If
$
\ZZ_+ (X_G) = \ZZ (X_G) \cap \QQ_+ (X_G) 
$
holds, then we have
$
\ZZ_+ (A_G) = \ZZ (A_G) \cap \QQ_+ (A_G) 
$.
\end{Proposition}

\begin{proof}
Suppose that $
\ZZ_+ (X_G) = \ZZ (X_G) \cap \QQ_+ (X_G) 
$
holds.
Let $ \xb \in \ZZ (A_G) \cap \QQ_+ (A_G)$.
Since $(0,\ldots,0,1) \in \ZZ (X_G) $, there exists an integer $\alpha$
such that 
$$\left(\begin{array}{c} \xb \\ \alpha \end{array}\right) \in \ZZ (X_G) \cap \QQ_+ (X_G) =\ZZ_+ (X_G) .$$
Thus $ \xb \in \ZZ_+ (A_G)$ as desired.
\end{proof}

\begin{Remark}
\label{k5isce}
{\em
The graph $K_5$ is a counterexample of the converse of Proposition \ref{nonhomo}.
}
\end{Remark}

\subsection{Main results}

The main purpose of the present paper is to prove that the set of graphs $G$
such that $\cpoly$ is normal is minor closed (Corollary \ref{mcmc}). 
Thanks to Corollary \ref{mcmc}, we have large classes of normal cut polytopes
(Theorem \ref{suspension}, Corollary \ref{coro} and Theorem \ref{just}).
In addition, in Section 4, we will show that, in order to study Conjecture 1.1,
it is enough to consider 4-connected plane triangulations.

Since the converse of Proposition \ref{nonhomo} is not true in general (Remark \ref{k5isce}),
we cannot apply the results on Hilbert bases to our problem {\em directly}.
However there are a lot of useful idea in \cite{Lau}.
For example, the idea of the proof of Theorem \ref{main} comes from that of
\cite[Proposition 2.3]{Lau} and 
the proof of Theorem \ref{glue} is similar to that of \cite[Proposition 2.7]{Lau}.

\section{Deletion of an edge}

Since the origin belongs to $A_G$, 
we have $(0,\ldots,0,1) \in  X_G$.
Hence it follows from \cite[p.258]{Lau} that, for $\xb \in \ZZ^E$ and $\alpha \in \ZZ$,
\begin{eqnarray}
\left(
\begin{array}{c}
\xb\\
\alpha
\end{array}
\right)
 \in \ZZ (X_G) \ \Longleftrightarrow \  \sum_{e \in C} x_e  \equiv 0 \ (\mbox{mod } 2)
\end{eqnarray}
for each cycle $C$ of $G$.
From now on, we always assume that $G$ has no $K_5$ minor.
Then the following Proposition is known.

\begin{Proposition}[\cite{BaMa}]
\label{barahona}
Let $G$ be a graph without $K_5$ minor.
Then $\cpoly $ is the solution set of the following linear inequalities:
$$
0 \leq x_e \leq 1, \ e \in E
$$
$$
\sum_{e \in F} x_e - \sum_{e \in C \setminus F} x_e \leq |F| -1
$$
where $C$ ranges over the induced cycles of $G$ and $F$ ranges over the
odd subsets of $C$.
\end{Proposition}

Thanks to Proposition \ref{barahona},
we have the following:

\begin{Corollary}
\label{qplus}
Let $G$ be a graph without $K_5$ minor.
For a vector $\xb \in \QQ^E$ and a nonnegative integer $\alpha$,
$\left( \begin{array}{c} \xb \\ \alpha \end{array} \right) \in \QQ_+ (X_G)$
if and only if
$$
0 \leq x_e \leq \alpha, \ e \in E
$$
$$
\sum_{e \in F} x_e - \sum_{e \in C \setminus F} x_e \leq  \alpha \ (|F| -1)
$$
where $C$ ranges over the induced cycles of $G$ and $F$ ranges over the
odd subsets of $C$.
\end{Corollary}

\begin{proof}
It follows from the following fact:
$$
\frac{1}{\ \alpha\ } \xb \in \cpoly 
\ \ 
\Longleftrightarrow 
\ \ 
\left(
\begin{array}{c}
\xb\\
\alpha
\end{array}
\right)
\in \QQ_+ (X_G)
$$
for $0 < \alpha \in \ZZ$ and $\xb \in \QQ^E$.
\end{proof}

By using the equation (1) together with Corollary \ref{qplus},
we have the following.

\begin{Theorem}
\label{main}
Let $G$ be a graph.
If $\cpoly$ is normal, then $\cpolytope (G \setminus e_0)$
is normal for any edge $e_0$ of $G$.
\end{Theorem}

\begin{proof}
The idea of the proof is obtained from that of \cite[Proposition 2.3]{Lau}.
Let $G' = G \setminus e_0$.
Note that $G$ and $G'$ have no $K_5$ minor.
Let $A_{G'} =\{\ab_1,\ldots,\ab_N\}$ and
$$
\left(
\begin{array}{c}
\xb\\
\alpha
\end{array}
\right)
=
\sum_{i=1}^N
q_i 
\left(
\begin{array}{c}
\ab_i\\
1\end{array}
\right)
\in 
\ZZ (X_{G'}) \cap \QQ_+ (X_{G'})
$$
where $0 < \alpha \in \ZZ$
and $0 \leq q_i \in \QQ$ for $1 \leq i \leq N$.
Since $\cpoly$ is normal, it is enough to show that
there exists a nonnegative integer $\gamma$ such that
$$
\left(
\begin{array}{c}
\gamma\\
\xb\\
\alpha
\end{array}
\right)
\in \ 
\ZZ (X_{G}) \cap \QQ_+ (X_{G}) \ = \ \ZZ_+ (X_G).
$$

Let 
$\xb' = 
\left(
\begin{array}{c}
\gamma\\
\xb
\end{array}
\right)
$
where
$\gamma \in \QQ$.
Thanks to Corollary \ref{qplus}, 
$
\left(
\begin{array}{c}
\xb'\\
\alpha
\end{array}
\right)
\in 
\QQ_+ (X_{G})
$
if and only if 
\begin{eqnarray}
 & & 0 \leq \gamma \leq \alpha\\
 & & \sum_{e \in F} x_e' - \sum_{e \in C \setminus F} x_e' \leq  \alpha \ (|F| -1)
\end{eqnarray}
where $C$ ranges over the induced cycles of $G$ with $e_0 \in C$ and $F$ ranges over the
odd subsets of $C$.
Then the equations (2) and (3) have a solution $\gamma$.
In fact, 
$$
\sum_{i=1}^N
q_i 
\left(
\begin{array}{c}
\delta_i\\
\ab_i\\
1\end{array}
\right)
=
\left(
\begin{array}{c}
\sum_{i=1}^N q_i \delta_i\\
\xb\\
\alpha
\end{array}
\right)
\in
\QQ_+ (X_G),
$$
where 
$A_G = 
\left\{
\left(
\begin{array}{c}
\delta_1\\
\ab_1
\end{array}
\right),
\ldots,
\left(
\begin{array}{c}
\delta_N\\
\ab_N
\end{array}
\right)
\right\}$.
Let
\begin{eqnarray*}
x_{\max} &=& \underset{(C, \ F) \ | \  e_0 \in C \setminus F}{\max}
\left( \sum_{e \in F} x_e' - \sum_{e \in C \setminus F, \ e \neq e_0} x_e' - \alpha \ (|F| -1) \right)
\in \ZZ ,\\
x_{\min} &=& \underset{(C, \ F) \ | \ e_0 \in F}{\min}
\left( - \sum_{e \in F, \ e \neq e_0} x_e' + \sum_{e \in C \setminus F} x_e' + \alpha \ (|F| -1) \right)
\in \ZZ .
\end{eqnarray*}
Note that $|F|-1$ is even.
By (2) and (3) above, we have
\begin{eqnarray}
\left(
\begin{array}{c}
\gamma\\
\xb\\
\alpha
\end{array}
\right)
\in 
\QQ_+ (X_{G}) \ 
\Longleftrightarrow \ 
\max(0, x_{\max} ) \leq \gamma \leq \min( \alpha, x_{\min}).
\end{eqnarray}
On the other hand, let $C$ be an arbitrary cycle of $G$ containing $e_0$.
Then by (1),
\begin{eqnarray}
\left(
\begin{array}{c}
\gamma\\
\xb\\
\alpha
\end{array}
\right)
\in 
\ZZ (X_{G}) \ 
\Longleftrightarrow \ 
\gamma \equiv \sum_{e \in C , \ e \neq e_0} x_e' \ (\mbox{\rm mod } 2).
\end{eqnarray}
If $\max(0, x_{\max} ) < \min( \alpha, x_{\min})$, then
$\max(0, x_{\max} ) +1 \leq  \min( \alpha, x_{\min})$
and hence either $\gamma=\max(0, x_{\max} )$ or $\gamma=\max(0, x_{\max} )+1$ satisfies
the conditions (4) and (5).
Suppose that 
$\max(0, x_{\max} ) = \min( \alpha, x_{\min})$.
Let $\gamma = \max(0, x_{\max} ) = \min( \alpha, x_{\min}) \in \ZZ$.
Since $0 < \alpha$, at least one of $\gamma = x_{\max}$ or $\gamma = x_{\min}$ holds.
If $\gamma = x_{\max}$, then there exists a cycle $C$ of $G$ containing $e_0$ such that
\begin{eqnarray*}
\gamma &=&  \sum_{e \in F} x_e' - \sum_{e \in C \setminus F, \ e \neq e_0} x_e' - \alpha \ (|F| -1)\\
 & \equiv & \sum_{e \in C, \ e \neq e_0} x_e' \ \ (\mbox{mod } 2).
\end{eqnarray*}
Similarly, if $\gamma = x_{\min}$, then there exists a cycle $C$ of $G$ containing $e_0$ such that
\begin{eqnarray*}
\gamma &=&   - \sum_{e \in F, \ e \neq e_0} x_e' + \sum_{e \in C \setminus F} x_e' + \alpha \ (|F| -1)\\
& \equiv & \sum_{e \in C, \ e \neq e_0} x_e' \ \ (\mbox{mod } 2).
\end{eqnarray*}
In both cases, $\gamma$ satisfies the conditions (4) and (5).
Thus we have 
$$
\left(
\begin{array}{c}
\gamma\\
\xb\\
\alpha
\end{array}
\right)
\in 
\ZZ (X_{G}) \cap \QQ_+ (X_G) = \ZZ_+ (X_G)
$$
and hence
$
\left(
\begin{array}{c}
\xb\\
\alpha
\end{array}
\right)
\in 
\ZZ_+ (X_{G'})
$
as desired.
\end{proof}

It is known \cite[Lemma 3.2 (2)]{StSu} that, if $\cpoly$ is normal
and $G'$ is obtained from $G$ by contracting an edge,
then $\cpolytope (G')$ is normal.
Thus,
we have the following:

\begin{Corollary}
\label{mcmc}
The set of graphs $G$ such that $\cpoly$ is normal is minor closed.
\end{Corollary}

\section{Clique sums and normality}

Let $G_1 = (V_1, E_1)$ and 
$G_2 = (V_2, E_2)$ be graphs such that $V_1 \cap V_2$ is a clique of both graphs.
The new graph $G = G_1 \sharp G_2$ with the vertex set $V = V_1 \cup V_2$
and edge set $E=E_1 \cup E_2$ is called the {\em clique sum} of $G_1$ and $G_2$
along $V_1 \cap V_2$.
If the cardinality of $V_1 \cap V_2$ is $k+1$, this operation is called a {\em $k$-sum} of the graphs.

\begin{Proposition}[\cite{StSu}]
Let $G = G_1 \sharp G_2$ be a $0$, $1$ or $2$-sum of $G_1$ and $G_2$.
Then the set of generators (or Gr\"obner bases) of the toric ideal $I_G$
of $\cpoly$ consists of
that of $I_{G_1}$ and $I_{G_2}$ together with some quadratic binomials.
\end{Proposition}

It turns out that this holds even for normality.

\begin{Theorem}
\label{glue}
Let $G = G_1 \sharp G_2$ be a $0$, $1$ or $2$-sum of $G_1$ and $G_2$.
Then the cut polytope of $G$ is normal if and only if
the cut polytope of $G_i$ is normal for $i = 1,2$.
\end{Theorem}

\begin{proof}
This is similar to the proof of \cite[Proposition 2.7]{Lau}.

Since $G_1$ and $G_2$ are induced subgraphs of $G$,
the ``only if" part follows from \cite[Lemma 3.2 (1)]{StSu}.

Suppose that the cut polytope of $G_i$ is normal for $i = 1,2$.
Let $\{i_1,\ldots,i_k\}$ $(1 \leq k \leq 3)$ denote the common vertices of
$G_1$ and $G_2$.
It is easy to see that
we can express $A_G$ as
\begin{eqnarray}
A_G= 
\{\delta_G (S) \ | \ 
i_1 \in S \subset [n]\} \subset \ZZ^E.
\end{eqnarray}

\noindent
{\bf Case 1.} $k =3$

\noindent
By (6), we have $A_G = A_G^{++} \cup A_G^{+-} \cup A_G^{-+} \cup A_G^{--}$ where
\begin{eqnarray*}
A_G^{++} &=& 
\left\{
\left.
\left(
\begin{array}{c}
\xb\\
\yb\\
{\bf z}_0
\end{array}
\right)
\ \right| \ 
\left(
\begin{array}{c}
\xb\\
{\bf z}_0
\end{array}
\right) \in A_{G_1}^{++}, \ 
\left(
\begin{array}{c}
\yb\\
{\bf z}_0
\end{array}
\right) \in A_{G_2}^{++}
\right\}, \ \ \ 
{\bf z}_0
=
\left(
\begin{array}{c}
0\\
0\\
0
\end{array}
\right)
\\
A_G^{+-} &=& 
\left\{
\left.
\left(
\begin{array}{c}
\xb\\
\yb\\
{\bf z}_1
\end{array}
\right)
\ \right| \ 
\left(
\begin{array}{c}
\xb\\
{\bf z}_1
\end{array}
\right) \in A_{G_1}^{+-}, \ 
\left(
\begin{array}{c}
\yb\\
{\bf z}_1
\end{array}
\right) \in A_{G_2}^{+-}
\right\}, \ \ \ 
{\bf z}_1
=
\left(
\begin{array}{c}
0\\
1\\
1
\end{array}
\right)
\\
A_G^{-+} &=& 
\left\{
\left.
\left(
\begin{array}{c}
\xb\\
\yb\\
{\bf z}_2
\end{array}
\right)
\ \right| \ 
\left(
\begin{array}{c}
\xb\\
{\bf z}_2
\end{array}
\right) \in A_{G_1}^{-+}, \ 
\left(
\begin{array}{c}
\yb\\
{\bf z}_2
\end{array}
\right) \in A_{G_2}^{-+}
\right\},\ \ \ 
{\bf z}_2
=
\left(
\begin{array}{c}
1\\
0\\
1
\end{array}
\right)\\
A_G^{--} &=& 
\left\{
\left.
\left(
\begin{array}{c}
\xb\\
\yb\\
{\bf z}_3
\end{array}
\right)
\ \right| \ 
\left(
\begin{array}{c}
\xb\\
{\bf z}_3
\end{array}
\right) \in A_{G_1}^{--}, \ 
\left(
\begin{array}{c}
\yb\\
{\bf z}_3
\end{array}
\right) \in A_{G_2}^{--}
\right\}, \ \ \ 
{\bf z}_3
=
\left(
\begin{array}{c}
1\\
1\\
0\\
\end{array}
\right)
\end{eqnarray*}
\begin{eqnarray*}
A_{G_i}^{++} &=&
\{\delta_{G_i} (S) \ | \ 
i_1,i_2,i_3 \in S \subset [n_i]\} \subset \ZZ^{E_i}\\
A_{G_i}^{+-} &=&
\{\delta_{G_i} (S) \ | \ 
i_1,i_2 \in S \subset [n_i], \ \  i_3 \notin S\} \subset \ZZ^{E_i}\\
A_{G_i}^{-+} &=&
\{\delta_{G_i} (S) \ | \ 
i_1,i_3 \in S \subset [n_i], \ \ i_2 \notin S\} \subset \ZZ^{E_i}\\
A_{G_i}^{--} &=&
\{\delta_{G_i} (S) \ | \ 
i_1 \in S \subset [n_i], \ \ i_2, i_3 \notin S\} \subset \ZZ^{E_i}.
\end{eqnarray*}
Let
$
{\small
\left(
\begin{array}{c}
\xb\\
\yb\\
p\\
q\\
r\\
\alpha
\end{array}
\right)
}
\in
\ZZ (X_G) \cap \QQ_+ (X_G)
$
for a positive integer $\alpha$.
Then we have
$$
{\small
\left(
\begin{array}{c}
\xb\\
p\\
q\\
r\\
\alpha
\end{array}
\right)
}
\in
\ZZ (X_{G_1}) \cap \QQ_+ (X_{G_1}) = \ZZ_+ (X_{G_1}),
\ 
{\small
\left(
\begin{array}{c}
\yb\\
p\\
q\\
r\\
\alpha
\end{array}
\right)
}
\in
\ZZ (X_{G_2}) \cap \QQ_+ (X_{G_2}) = \ZZ_+ (X_{G_2})
.$$
Hence
\begin{eqnarray}
{\small
\left(
\begin{array}{c}
\xb\\
p\\
q\\
r\\
\alpha
\end{array}
\right)
}
&=&
\left(
\begin{array}{c}
\xb^{(1)}\\
{\bf z}_{k_1}\\
1
\end{array}
\right)
+\left(
\begin{array}{c}
\xb^{(2)}\\
{\bf z}_{k_2}\\
1
\end{array}
\right)
+\cdots+
\left(
\begin{array}{c}
\xb^{(\alpha)}\\
{\bf z}_{k_\alpha}\\
1
\end{array}
\right)
\mbox{ where }
\left(
\begin{array}{c}
\xb^{(i)}\\
{\bf z}_{k_i}
\end{array}
\right)
 \in A_{G_1} \ \ \ 
\\
{\small
\left(
\begin{array}{c}
\yb\\
p\\
q\\
r\\
\alpha
\end{array}
\right)
}
&=&
\left(
\begin{array}{c}
\yb^{(1)}\\
{\bf z}_{k_1'}\\
1
\end{array}
\right)
+\left(
\begin{array}{c}
\yb^{(2)}\\
{\bf z}_{k_2'}\\
1
\end{array}
\right)
+\cdots+
\left(
\begin{array}{c}
\yb^{(\alpha)}\\
{\bf z}_{k_\alpha'}\\
1
\end{array}
\right)
\mbox{ where }
\left(
\begin{array}{c}
\yb^{(j)}\\
{\bf z}_{k_j'}
\end{array}
\right)
 \in A_{G_2}
.\ \ \ 
\end{eqnarray}
Let $\xi_i$ (resp. $\xi_i'$) denote 
the number of ${\bf z}_i$ appearing in (7) (resp.~(8))
for each $i = 0,1,2,3$.
Then we have
$p = \xi_2 + \xi_3 =\xi_2' + \xi_3'$,
$q = \xi_1 + \xi_3 =\xi_1' + \xi_3'$,
$r = \xi_1 + \xi_2 =\xi_1' + \xi_2'$,
and
$\alpha = \sum_{i=0}^4 \xi_i =\sum_{i=0}^4 \xi_i'$.
Hence $\xi_i = \xi_i'$ 
 for all $i = 0,1,2,3$.
Thus, by changing the numbering, we have
$$
{\small
\left(
\begin{array}{c}
\xb\\
\yb\\
p\\
q\\
r\\
\alpha
\end{array}
\right)
}
=
\left(
\begin{array}{c}
\xb^{(1)}\\
\yb^{(1)}\\
{\bf z}_{k_1}\\
1
\end{array}
\right)
+\left(
\begin{array}{c}
\xb^{(2)}\\
\yb^{(2)}\\
{\bf z}_{k_2}\\
1
\end{array}
\right)
+\cdots+
\left(
\begin{array}{c}
\xb^{(\alpha)}\\
\yb^{(\alpha)}\\
{\bf z}_{k_\alpha}\\
1
\end{array}
\right)
\in \ZZ_+ (X_{G}).
$$

\bigskip

\noindent
{\bf Case 2.} $k =1,2$

\noindent
By (6), if $k=1$, then
$A_G = 
\left\{
\left.
\left(
\begin{array}{c}
\xb\\
\yb
\end{array}
\right)
\ \right| \ 
\xb \in A_{G_1}, \ 
\yb \in A_{G_2}
\right\}
$
and if $k=2$, then
we have $A_G = A_G^+ \cup A_G^-$ where
\begin{eqnarray*}
A_G^+ &=& 
\left\{
\left.
\left(
\begin{array}{c}
\xb\\
\yb\\
0
\end{array}
\right)
\ \right| \ 
\left(
\begin{array}{c}
\xb\\
0
\end{array}
\right) \in A_{G_1}^+, \ 
\left(
\begin{array}{c}
\yb\\
0
\end{array}
\right) \in A_{G_2}^+
\right\}\\
A_G^- &=& 
\left\{
\left.
\left(
\begin{array}{c}
\xb\\
\yb\\
1
\end{array}
\right)
\ \right| \ 
\left(
\begin{array}{c}
\xb\\
1
\end{array}
\right) \in A_{G_1}^-, \ 
\left(
\begin{array}{c}
\yb\\
1
\end{array}
\right) \in A_{G_2}^-
\right\}\\
A_{G_i}^+ &=&
\{\delta_{G_i} (S) \ | \ 
i_1,i_2 \in S \subset [n_i]\} \subset \ZZ^{E_i}\\
A_{G_i}^- &=&
\{\delta_{G_i} (S) \ | \ 
i_1 \in S \subset [n_i], i_2 \notin S\} \subset \ZZ^{E_i}.
\end{eqnarray*}
In both cases, the desired conclusion follows from the similar
(and simpler) argument in Case 1.
\end{proof}

A graph $G = (V,E)$ is called {\em edge-maximal without ${\cal H}$ minor},
if $G$ has no ${\cal H}$ minor but any graph $G' = (V,E')$ with
$E' = E \cup \{e\}$ and $e \notin E$ has ${\cal H}$ minor.

Let $G$ be a graph with vertex set $V = [n] = \{1,\ldots,n\}$ and edge set $E$.
The {\em suspension} of the graph $G$ is the new graph $\widehat{G}$
whose vertex set equals $[n+1] = V \cup \{n+1\}$ and whose edge set
equals $E \cup \{\{ i,n+1\} \ | \ i \in V\} $.
A cut ideal $I_{\widehat{G}}$ corresponds to the toric ideal arising from
the binary graph model of $G$.

\begin{Theorem}
\label{suspension}
Let $G$ be a graph.
Then $\cpolytope (\widehat{G})$ is normal if and only if
$G$ has no $K_4$ minor.
\end{Theorem}

\begin{proof}
If $G$ has $K_4$ minor, then $\widehat{G}$ has $K_5$ minor.
Hence $\cpolytope (\widehat{G})$ is not normal.

It is known \cite[Proposition 7.3.1]{Die} that
a graph with at least three vertices is
edge-maximal without $K_4$ minor if and only if 
it is $1$ sum of $K_3$'s.
Hence, if $G$ is edge-maximal without $K_4$ minor,
then $\widehat{G}$ is $2$ sums of $K_4$'s.
Since the cut polytope of $K_4$ is normal,
$\cpolytope (\widehat{G})$ is normal by Theorem \ref{glue}.
Thus for any subgraph $G'$ of $G$, $\cpolytope (\widehat{G'})$ is 
normal by Theorem \ref{main}.
\end{proof}

\begin{Remark}
{\em
One of the referees pointed out that
Theorem \ref{suspension} implies the main result of \cite{Sul2}.
}
\end{Remark}

\begin{Example}
{\em
The cut polytope of a wheel graph $W_n = \widehat{C_n}$ is normal
since the cycle $C_n$ has no $K_4$ minor.
}
\end{Example}

By considering the subgraph of the graphs appearing in
Theorem \ref{suspension}, we have

\begin{Corollary}
\label{coro}
If $G$ has a vertex $v$ such that 
the induced subgraph of $G$ on $V \setminus \{v\}$
has no $K_4$ minor, then $\cpoly$ is normal.
\end{Corollary}

\begin{Example}
{\em
Let $G$ be a graph with $\leq 5$ vertices.
Then the cut polytope of $G$ is normal
if and only if $G \neq K_5$.
}
\end{Example}

\begin{Theorem}
\label{just}
Let $G$ be a graph with no $K_5 \setminus e$ minor.
Then $\cpoly$ is normal.
\end{Theorem}

\begin{proof}
It is known 
\cite[p.180]{Die2}
that,
if $G$ is edge-maximal graph without 
$K_5 \setminus e$ minor,
then $G$ is obtained by $1$-sum of the graphs
$K_3$, $K_{3,3}$, $W_n$, and the prism $C_3 \times K_2$.
Since the cut polytope of all of them are normal,
$\cpoly$ is normal by Theorem \ref{glue}.
By Theorem \ref{main}, the cut polytope of any subgraph of $G$ 
is normal. 
\end{proof}

\section{Sturmfels--Sullivant Conjecture }

Although Conjecture \ref{concon} is still open,
the following is known \cite[p.181]{Die2}
in graph theory.

\begin{Proposition}
Let $G$ be an edge-maximal graph without $K_5$ minor.
If $G$ has at least $3$ vertices, then $G$ 
is $1$ or $2$ sum of $K_3$, $K_4$, 4-connected plane triangulations and the graph $V_8$.
\end{Proposition}

The cut polytopes of $K_3$ and $K_4$ are normal.
Moreover,

\begin{Example}
{\em
Let $V_8$ be the graph
with the edge set
$$
\{
\{1,2\},
\{2,3\},
\{3,4\},
\{4,5\},
\{5,6\},
\{6,7\},
\{7,8\},
\{1,8\},
\{1,5\},
\{2,6\},
\{3,7\},
\{4,8\}
\}.
$$
Since $V_8$ has an induced cycle of length $5$,
$\cpolytope (V_8)$ is not compressed by \cite[Theorem 3.2]{Sul}.
It follows from Corollary \ref{coro} that
the cut polytope of any proper minor of $V_8$ is normal.
By the software {\tt Normaliz} \cite{BrIc}, we can check that $\cpolytope (V_8)$ is normal. 
}
\end{Example}

Thus, in order to prove Conjecture \ref{concon}, it is enough to prove 
one of the following
conjectures:

\begin{Conjecture}
The cut polytope $\cpoly$ is normal if
$G$ is a 4-connected plane triangulation.
\end{Conjecture}

\begin{Conjecture}
The cut polytope $\cpoly$ is normal if
$G$ is a grid graph.
\end{Conjecture}

\bigskip

\noindent
{\bf Acknowledgements.}\ 
The results on this paper were obtained
while the author was visiting University of Washington
from August to December 2008.
He appreciates the warm hospitality he received from Rekha R. Thomas and 
people of Department of Mathematics,
University of Washington.
In addition, the author thanks Rekha R. Thomas for 
introducing
cut polytopes to him
together with important references and useful discussions.
This research was supported by JST, CREST.

\bigskip

Department of Mathematics, College of Science, Rikkyo University,
Toshima-ku, Tokyo 171-8501, Japan.
{\tt ohsugi@rkmath.rikkyo.ac.jp}

\end{document}